\documentclass[a4paper,11pt,leqno]{amsart}
\usepackage{subfig}
\usepackage{etex}
\usepackage[a4paper,left=25mm,right=25mm,top=30mm,bottom=30mm,marginpar=25mm]{geometry}
\usepackage[utf8x]{inputenc}
\usepackage[english]{babel}
\usepackage{amsmath}
\usepackage{amssymb}
\usepackage{amsthm}
\usepackage{mathabx}
\usepackage{bbm}
\usepackage{enumerate}
\usepackage{empheq}
\usepackage{fancybox}
\usepackage{xcolor}
\usepackage{mathrsfs}
\usepackage{slashed}
\usepackage{tikz}
\usetikzlibrary{shapes,snakes,arrows}
\usetikzlibrary{positioning}
\usepackage{harpoon}
\usepackage{stmaryrd}
\usepackage{setspace}
\usepackage{fancybox}
\usepackage{esint}
\usepackage{bm}
\usepackage{hyperref}
\usepackage{mathtools}
\DeclareMathOperator{\son}{SO}
\DeclareMathOperator{\pst}{1^{*}}
\DeclareRobustCommand{\rchi}{{\mathpalette\irchi\relax}}
\newcommand{\irchi}[2]{\raisebox{\depth}{$#1\chi$}}
\tikzset{node distance=2cm, auto}
\usetikzlibrary{patterns}

\usepackage[all]{xy}
\theoremstyle{definition}
\newtheorem{definition}{Definition}
\theoremstyle{plain}
\newtheorem{theorem}{Theorem}
\newtheorem{lemma}{Lemma}
\newtheorem{proposition}[lemma]{Proposition}

\theoremstyle{remark}
\newtheorem{remark}{Remark}

\newcommand{\rB}[1]{\ensuremath{\left(#1\right)}}

\newcommand{\sB}[1]{\ensuremath{\left[#1\right]}}
\newcommand{\cB}[1]{\ensuremath{\left\{#1\right\}}}
\newcommand{\scal}[2]{\ensuremath{\left\langle#1,#2\right\rangle}}

\newcommand{\norm}[1]{\left|\left|#1\right|\right|}
\newcommand{\modulus}[1]{\left|#1\right|}

\renewcommand\epsilon{\varepsilon}
\renewcommand\rho{\varrho}

\newcommand{\de}{\ensuremath{\mathrm{d}}}

\newenvironment{frcseries}{\fontfamily{frc}\selectfont}{}

\DeclareMathOperator{\Curl}{Curl}
\DeclareMathOperator{\spt}{spt}
\DeclareMathOperator{\dist}{dist}

\renewcommand{\tilde}{\widetilde}

\newcommand{\SOn}{SO(n)}

\newcommand{\zak}{%
  \mathbin{\vrule height 1.6ex depth 0pt width
0.13ex\vrule height 0.13ex depth 0pt width 1.3ex}
}

\tikzset{node distance=2cm, auto}

\renewcommand{\phi}{\varphi}

\renewcommand{\phi}{\varphi}

\begin{document}
\title[Geometric Rigidity for Incompatible Fields]{Geometric Rigidity Estimates for Incompatible Fields in dimension $\ge 3$}
\author{Gianluca Lauteri}
\address{Max-Planck-Institut f\"ur Mathematik in den Naturwissenschaften, Leipzig, Germany}
\email[G.~Lauteri]{Gianluca.Lauteri@mis.mpg.de}
\author{Stephan Luckhaus}
\address{Institut f\"ur Mathematik, Leipzig University, D-04009 Leipzig, Germany}
\email[S.~Luckhaus]{\tt Stephan.Luckhaus@math.uni-leipzig.de}
\maketitle

\begin{abstract}
  We prove geometric rigidity inequalities for incompatible fields in dimension higher than $2$. We are able to obtain strong scaling-invariant $L^p$ estimates in the supercritical regime $p > 1^* = \frac{n}{n-1}$, while for critical exponent $1^*$ we have a scaling invariant inequality only for the weak $L^{1^*}$ norm. Although not optimal, such an estimate in $L^{1^*, \infty}$ is enough in order to infer a useful lemma which gives $BV$ bounds for $\son(n)$-valued fields with bounded $\Curl$.
\end{abstract}

\section{Introduction}

The geometric rigidity estimate for gradient fields proved in~\cite{FJM} plays a crucial role in nonlinear elasticity theory. However, in the study of lattices with dislocations, a geometric rigidity estimate for incompatible fields (i.e., fields not arising from gradients) becomes necessary (cf. e.g.~\cite{MSZ} and~\cite{LL}). In~\cite{MSZ}, the authors proved a (\emph{scaling invariant}) version of the geometric rigidity theorem in~\cite{FJM} for incompatible fields in dimension $2$ for the critical exponent.\\
In this work we give a proof of the analogous result in dimension $\ge 3$, for the supercritical regime $p > 1^* = \frac{n}{n-1}$ (Theorem~\ref{thm:rig_LL_1}). The approach is to write down an incompatible field as the sum of a compatible term, for which we can use the classical geometric rigidity from~\cite{FJM} and a remainder, which is the $L^p$ norm of a weakly singular operator (the \emph{averaged linear homotopy operator}), whose derivative is a Calder\'on-Zygmund operator. This allows to give the bounds in the supercritical case. On the other hand, for the critical exponent we can still use the weak geometric rigidity estimate proved in~\cite{CDM} in order to find a scaling invariant estimate for the weak-$L^{1^*}$ norm (Theorem~\ref{thm:useless1}). From Theorem~\ref{thm:useless1}, we deduce directly in Proposition~\ref{prop:curl_bounds_D_SOn} that the $\Curl$ of a matrix field $A \in L^{1^*, \infty}(\Omega)^{n\times n}$ (where $\Omega$ is an open bounded set in $\mathbb{R}^n$) taking values in $SO(n)$ bounds its gradient.

\section{Notations and Preliminaries}
In what follows, $C$ will denote a (universal) constant whose value is allowed to change from line to line. We put $\widehat{x}:=\frac{x}{\modulus{x}}$, while $L^p(U, \Lambda^r)$ ($W^{m, p}(U, \Lambda^r)$) denotes the space of $r$-forms on $U$ whose coefficients are $L^p$ ($W^{m,p }$) functions. Moreover, recall that we can identify a tensor field $A\in L^1(\Omega)^{n\times n}$ with a vector of $1$-forms of length $n$, that is with $\omega:=\rB{\omega^i}_{i=1}^n$, $\omega^i = A^i_j \de x^j$, and its $\Curl$ with $\de \omega$ (or, more precisely, with $\rB{\star \de \omega}^{\flat}$), given by \[\displaystyle \de \omega^i = \sum_{j < k} \rB{\frac{\partial A^i_j}{\partial x^k} - \frac{\partial A^i_k}{\partial x^j}} \de x^j \wedge \de x^k .\]
We recall that a  real-valued function $f$ from a measure space $(X, \mu)$ is in $L^{p, \infty}(X, \mu)$ or ($L^p_w(X, \mu)$) if
	\[
	 \norm{f}_{L^{p, \infty}(X, \mu)}:=\sup_{t > 0} t \mu\rB{\cB{x \in X\biggr| \modulus{f(x)} > t}}^{\frac{1}{p}} < \infty.
	\]
	Is easy to check that $\norm{\cdot}_{L^{p, \infty}}(X, \mu)$ is only a quasi-norm, that is the triangle inequality holds just in the weak form
	\[
	 \norm{f+g}_{L^{p, \infty}(X, \mu)} \le C_{p} \rB{\norm{f}_{L^{p, \infty}(X, \mu)} + \norm{g}_{L^{p, \infty}(X, \mu)}}.
	\]
	We write $L^{p, \infty}(\Omega)$ for $L^{p, \infty}(\Omega, \modulus{\cdot})$, when $\Omega\subset \mathbb{R}^n$ and $\modulus{\cdot}$ is the Lebesgue measure.\\
We recall the
	\begin{definition}
		Let $U \subset \mathbb{R}^n$ be a star-shaped domain with respect to the point $y \in U$. The \emph{linear homotopy operator} at the point $y$ is the operator
		\[
			k_y= k_{y, r} : \Omega^r(U) \to \Omega^{r-1}(U),
		\]
		defined as
		\[
			(k_y \omega) (x):=\int_0^1{s^{r-1} \omega(sx + (1-s)y)\zak (x-y)\de s},
		\]
	\end{definition}
	where $(\omega(x)\zak v)\sB{v_1, \cdots v_{n-1}}:=\omega(x)\sB{v, v_1, \cdots, v_{n-1}}$. It is well known that the linear homotopy operator satisfies
	\begin{equation}
	 \label{eq:lho1}
	 \omega = k_{y, r+1} \de \omega + \de k_{y, r} \omega\quad \forall \omega \in \Omega^r(U).
	\end{equation}
	In order to get more regularity, we consider the following \emph{averaged} linear homotopy operator on $B := B(0, 1)$, which coincides with the one introduced by Iwaniec and Lutoborski in~\cite{IL}, except for the choice of the weight function:
	\[
	 T = T_r : \Omega^r(B) \to \Omega^{r-1}(B),
	\]
	\[
	 T\omega(x):=\int_B \phi(y) \rB{k_y\omega}(x) \de y,
	\]
	where $\phi \in \mathcal{C}^{\infty}_c(B(0, 2))$ is a positive cut-off function, with $\phi \equiv 1$ in $B$ and 
	\[\max\cB{\norm{\phi}_{L^{\infty}(\mathbb{R}^n)}, \norm{\nabla \phi}_{L^{\infty}(\mathbb{R}^n)}}\le 3.\]
	Clearly,~\eqref{eq:lho1} holds for $T$ as well:
	\begin{equation}
	 \label{eq:lho2}
	 \omega = T\de \omega + \de T \omega.
	\end{equation}
An essential result is the rigidity estimate due to Friesecke, James and M\"uller: 
\begin{theorem}[~\cite{FJM}]
 \label{thm:FJM}
 Let $\Omega\subset\mathbb{R}^n$ be a bounded Lipschitz domain, $n \ge 2$, and let $1<p<\infty$. There exists a constant $C = C(p, \Omega)$ such that for every $u \in W^{1, 2}(\Omega)$ there exists a rotation $R \in SO(n)$ such that
 \[
  \norm{\nabla u - R}_{L^p(\Omega)^{n\times n}} \le C\norm{\dist(\nabla u, SO(n))}_{L^p(\Omega)^{n\times n}}.
 \]
\end{theorem}
For weak-$L^p$ estimate, we shall need the following theorem proved by Conti, Dolzmann and M\"uller:
	\begin{theorem}[~\cite{CDM}]
	 \label{thm:cdm}
	 Let $p \in (1, \infty)$ and $\Omega\subset\mathbb{R}^n$ be a bounded connected domain. There exists a constant $C > 0$ depending only on $p, n$ and $\Omega$ such that for every $u \in W^{1, 1}(\Omega)^n$ such that $\dist(\nabla u, SO(n)) \in L^{p, \infty}(\Omega)^{n\times n}$ there exists a rotation $R \in SO(n)$ such that
	 \begin{equation}
	  \label{eq:cdm}
	  \norm{\nabla u - R}_{L^{p, \infty}(\Omega)^{n \times n}} \le C \norm{\dist(\nabla u, SO(n))}_{L^{p, \infty}(\Omega)^{n\times n}}.	 
	 \end{equation}
	\end{theorem}
	We also recall that, as proved in~\cite{IL}, $T$ satisfies (for smooth forms $\omega$) the pointwise bound
	\begin{equation}
	 \label{eq:bound_LHO}
	 \modulus{T\omega(x)} \le C_{n, r} \int_B \frac{\modulus{\omega(y)}}{\modulus{x-y}^{n-1}} \de y.
	\end{equation}
	Indeed, for $\omega = \omega_{\alpha} \de x^{\alpha}\in \Omega^r(B)$ we have
	\[
	 T\omega (x) = \rB{\int_B \de y \phi(y) \int_0^1 t^{r-1} \scal{x-y}{e_i} \omega_{\alpha}(tx+(1-t)y)} \de x^{\alpha} \zak e_i.
	\]
	We then make the substitution $\Phi(y, t):=\rB{tx+(1-t)y, \frac{t}{1-t}} \equiv (z(t, y), s(t))$, $\Phi: B(0, 1)\times (0, 1)\to B(0,1) \times (0, \infty)$, which gives
	\[
	 \begin{split}
	 T\omega(x) 
	            &= \rB{\int_B \de z \omega_{\alpha}(z)\frac{\scal{x-z}{e_i}}{\modulus{x-z}^n} \int_0^2 s^{r-1}(1+s)^{n-r}\phi(z+s\widehat{z - x})} \de x^{\alpha}\zak e_i \equiv\\
                    &\equiv \rB{\int_B  K^i_r(z, x-z) \omega_{\alpha}(z) \de z} \de x^{\alpha} \zak e_i,
	 \end{split}
	\]
        where
        \[
		K^i_r(x, h):=\frac{\scal{h}{e_i}}{\modulus{h}^n} \int_0^2 s^{r-1} (1+s)^{n-r} \phi(x-s\widehat{h}) \de s,
        \]
	and we noticed that, since $\phi$ has compact support, the integral from $0$ to $\infty$ actually reduces to an integral over a finite interval. That is, we get~\eqref{eq:bound_LHO}. It also follows easily from~\eqref{eq:bound_LHO} that $T$ is a compact operator from $L^p(B, \Lambda^r)$ to $L^p(B, \Lambda^{r-1})$. Moreover, by density,~\eqref{eq:lho2} extends to every differential form $\omega \in W^{1, p}(B, \Lambda^r)$, and to every differential form $\omega \in L^1(B, \Lambda^r)$ whose differential is a bounded Radon measure, $\de \omega \in \mathcal{M}_b(B, \Lambda^{r+1})$.
\section{Proof of the Main Results}
	Using the homotopy operator, we get the following weak-$L^p$ geometric rigidity estimate for incompatible fields:
	\begin{theorem}
   \label{thm:useless1}
	 Let $1^* = 1^*(n):=\frac{n}{n-1}$, and let $B \subset \mathbb{R}^n$ be the unit ball of $\mathbb{R}^n$. 
	There exists a constant $C = C(n) >0 $ such that for every $A \in L^{p^{*}}(B)$ whose $\Curl(A)$ is a vector measure on $U$ with bounded total variation and whose support is contained in $B$, i.e. $\spt \Curl(A) \Subset B$, there exist a rotation $R \in SO(n)$ such that
	 \[
	  \norm{A - R}_{L^{1^*, \infty}(B)} \le C\rB{\norm{\dist(A, SO(n))}_{L^{1^*, \infty}(B)} + \modulus{\Curl(A)}(B)}.
	 \]
	\end{theorem}
	\begin{proof}
	 Take any measurable subset $E \subset B$, and let $r > 0$ be such that $\modulus{B(0, r)} = \modulus{E}$. Then, using~\eqref{eq:bound_LHO} and the Hardy-Littlewood inequality
	 \[
	  \begin{split}
	   \int_E \de x \modulus{(T\omega)(x)} & \le C \int_E \de x \int_B \de y \frac{\modulus{\omega(y)}}{\modulus{x - y}^{n-1}} = \\
	  &=C \int_B \de y \modulus{\omega(y)} \int_E \frac{\de x}{\modulus{x - y}^{n-1}} \le\\
	   &\le C \int_B \de y \modulus{\omega(y)} \int_{\mathbb{R}^n} \rchi_{E - x}(y)\frac{\de y}{ \modulus{y}^{n-1}} \le\\
	   &\le C \int_B \de y \modulus{\omega(y)} \int_{\mathbb{R}^n} \rchi_{B(0, r)} \frac{\de y}{\modulus{y}^{n-1}} \le \\
	   &= C \int_B \de y \modulus{\omega(y)} \int_0^r \de t \int_{\partial B(0, t)} \frac{\de y}{t^{n-1}} = \\
	   &= C r \norm{\omega}_{L^1(B)} = C \modulus{E}^{\frac{1}{n}} \norm{\omega}_{L^1(B)}.
	  \end{split}
	 \]
	 This gives immediately
	 \[
	  \norm{T\omega}_{L^1(B)} \le C_n \norm{\omega}_{L^1(B)},
	 \]
	 and thus, using~\eqref{eq:lho2}, $\norm{A - T\de A}_{L^1(B)} \le C \norm{\de A}_{L^1(B)}$, which extends immediately by density in the case when $\de A$ is a vector measure with bounded total variation. Choosing $E = \cB{x \in B \biggr| \modulus{T\omega(x)} > t}$, for $t > 0$
	 \[
	  t \modulus{E} \le \int_E \modulus{T\omega(x)} \de x \le C \modulus{E}^{\frac{1}{n}} \modulus{\de A}(B).
	 \]
	 Passing to the supremum over $t > 0$, we find
            \begin{equation}
                \label{eq:weak_Lp_lho}
	  \norm{T\de A}_{L^{1^*, \infty}(B)} \le C_n \modulus{\de A}(B).
            \end{equation}
	 Since $B$ is convex and $d(A - T\de A) = \de^2 TA = 0$, we can find a function $g$ such that $\de g = A - T\de A$. From the estimates proven, is possible to apply Theorem~\ref{thm:cdm} to $g$ and find
	 \[
	  \norm{\de g - R}_{L^{1^*, \infty}(B)} \le C \norm{\dist(\de g, SO(n))}_{L^{1^*, \infty}(B)}.
	 \]
	 But
	 \[
	  \norm{\de g - R}_{L^{1^*, \infty}(B)} \ge C \norm{A - R}_{L^{1^*, \infty}(B)} - \norm{T \de A}_{L^{1^*, \infty}(B)}
	 \]
	 and
	 \[
	  \norm{\dist(\de g, SO(n))}_{L^{1^*, \infty}(B)} \le \norm{\dist(A, SO(n))}_{L^{1^*, \infty}(B)} + \norm{T\de A}_{L^{1^*, \infty}(B)}.
	 \]
	 In particular,
	 \[
	  \norm{A - R}_{L^{1^*, \infty}(B)} \le C\rB{\norm{\dist(A, SO(n))}_{L^{1^*, \infty}(B)} + \modulus{\Curl(A)}(B)}.\qedhere
	 \]
\end{proof}

We now give another estimate for $L^p$ norms. It requires an $L^{\infty}$-bound on the matrix field $A$, which is natural in the context of the theory of elasticity.
\begin{theorem}
    \label{thm:rig_LL_1}
    Let $n \ge 3$, $1^*:=1^*(n):=\frac{n}{n-1}$, $p \in [1^*, 2]$ and fix $M > 0$. There exists a constant $C = C(n, M, p) > 0$, depending only on the dimension $n$, the exponent $p$ and the constant $M$, such that for every $A \in L^{\infty}(B)$, with $\norm{A}_{\infty} \le M$ and $\Curl(A) \in \mathcal{M}_b(B, \Lambda^2)$, $B:=B(0, 1)$, there exists a corresponding rotation $R \in \SOn$ for which, if $p > 1^*$
    \begin{equation}
	    \label{eq:rig_LL_p}
            \int_B \modulus{A - R}^p \de x \le C\rB{\int_B\dist^{p}(A, \SOn)\de x + \modulus{\Curl(A)}^{1^*}(B)},
    \end{equation}
   while, if $p = 1^*$,
    \begin{equation}
        \label{eq:rig_LL}
        \begin{split}
            \int_B \modulus{A - R}^{1^*} \de x \le& C\int_B\dist^{1^*}(A, \SOn)\de x + \\
                                                  & + C\modulus{\Curl(A)}^{1^*}(B)\cB{\modulus{\log\rB{\modulus{\Curl(A)}(B)}}+1}.
        \end{split}
    \end{equation}
\end{theorem}
\begin{remark}
	\label{rmk:everything_pointless}
    The constant $C$ in~\eqref{eq:rig_LL} is \emph{not} scaling invariant in the critical regime $p = 1^*$.
\end{remark}
\begin{proof}[Proof of Theorem~\ref{thm:rig_LL_1}]
	Without loss of generality, we can assume $T\de A$ not identically constant. Indeed, if $T\de A$ is identically constant, from the identity $T\de A = A - \de T A$, we see that $\de A = 0$, hence the result follows applying Theorem~\ref{thm:FJM}. As in the proof of Theorem~\ref{thm:useless1}, applying Theorem~\ref{thm:FJM} (and using $\modulus{a - b}^p \ge 2^{1-p}\modulus{a}^p - \modulus{b}^p$) we find a rotation $R \in \SOn$ for which the inequality
    \begin{equation}
        \label{eq:useless11}
        \int_B \modulus{A - R}^p \de x \le C_n\rB{\int_B\modulus{\dist(A, \SOn)}^p\de x + \int_B \modulus{T\de A(x)}^p \de x}
    \end{equation}
    holds. We then just need to estimate the last term in the right hand side of~\eqref{eq:useless11}. For, fix a $\Lambda > 1$ (to be chosen later), and define the integrals
    \[
        I:=\int_{\modulus{T\de A } > \Lambda} \modulus{T\de A}^p \de x,\qquad II:=\int_{\modulus{T\de A}\le \Lambda} \modulus{T \de A}^p \de x.
    \]
    We now give an estimate for $I$. Firstly, we recall that $T$ is a bounded operator from $L^p(B, \Lambda^r)$ into $W^{1, p}(B, \Lambda^{r+1})$, whenever $p\in (1, \infty)$ (cf. \cite[Proposition 4.1]{IL}). Moreover, $T\de A = A - \de TA $, and $\nabla T = S_1 + S_2$, where $S_1$ is a ``weakly'' singular operator which maps continuously $L^{\infty}$ into itself, while $S_2$ is a Calder\'on-Zygmund operator (cf. \cite[Proposition 4.1]{IL}). In particular,
    \[
        \norm{T \de A}_{\text{BMO}}\le C_n \norm{A}_{\infty} \le C_n M,
    \]
where $C_n > 0$ is a constant depending only on the dimension. Now, we can write
    \begin{equation}
        \label{eq:useless777}
	I = \Lambda^{p - \pst}\Lambda^{\pst} \modulus{\cB{\modulus{T \de A} > \Lambda}} + I',\qquad I':=\int_{\Lambda}^{\infty} \lambda^{p - 1}\modulus{\cB{\modulus{T\de A} > \lambda}} \de \lambda.
    \end{equation}
    Clearly,
    \[
        \Lambda^{\pst} \modulus{\cB{\modulus{T \de A} > \Lambda}} \le \norm{T\de A}_{L^{\pst, \infty}}^{\pst} \le C \modulus{\de A}(B)^{\pst}.
    \]
    We now take a Calder\'on-Zygmund decomposition of $F(x):=\modulus{T\de A(x)}^p$: namely, we find a function $g \in L^{\infty}$, with $\norm{g}_{\infty} \le 2^{-n}\Lambda^p$ and disjoint cubes $\cB{Q_j}_{j \ge 1}$ such that, if $b:=\sum_{j \ge 1} \rchi_{Q_j} F$,
    \[
        \begin{cases}
            F = g + b,\\
            2^{-n} \Lambda^p < \fint_{Q_j} F \de x\le \Lambda^p \quad \rB{\text{Jensen }\Rightarrow \modulus{\fint_{Q_j} T\de A(x) \de x} \le \Lambda},\\
            \modulus{\bigcup_{j \ge 1} Q_j} < \frac{2^n}{\Lambda^p} \int \modulus{T \de A}^p \de x.
        \end{cases}
    \]
With such a decomposition, outside the cubes $Q_j$, $\modulus{T \de A}^p = \modulus{g(x)} \le 2^{-n}\Lambda^p \le \Lambda^p$. Hence, using the John-Nirenberg inequality and the elementary estimate
    \[
        \int_x^{\infty} \lambda^q e^{-\lambda} \de \lambda \le e^{-x}(1+x),\quad \forall q \le 1 \text{ and } x \ge 1,
    \]
    we find that (provided $p \le 2$)
    \begin{equation}
        \label{eq:useless888}
        \begin{split}
            I' &= \int_{\Lambda}^{\infty} \lambda^{p - 1} \sum_{j \ge 1} \modulus{\cB{x \in Q_j \biggr| \modulus{T\de A} > \lambda}} \de \lambda \le \\
            &\le \int_{\Lambda}^{\infty} \lambda^{p - 1} \sum_{j \ge 1} \modulus{\cB{x \in Q_j\biggr| \modulus{T\de A(x) - \fint_{Q_j} T\de A \de x} > \lambda - \Lambda}} \de \lambda \le \\
            &\le C_1 \int_{\Lambda}^{\infty} \lambda^{p - 1} \rB{\sum_{j \ge 1} \modulus{Q_j}} \exp\rB{-C_2 \frac{\lambda - \Lambda}{\norm{T\de A}_{\text{BMO}}}} \de \lambda < \\
            &< C_1 \frac{2^n}{\Lambda^p} \rB{\int \modulus{T\de A}^p} e^{C_2 \frac{\Lambda}{\norm{T\de A}_{\text{BMO}}}} \rB{\frac{\norm{T\de A}_{\text{BMO}}}{C_2}}^{p} \int_{\frac{C_2}{\norm{T\de A}_{\text{BMO}}} \Lambda}^{\infty}\lambda^{p - 1} e^{-\lambda} \de \lambda \le \\
            &\le C_1 \frac{2^n}{\Lambda^p} \rB{\int \modulus{T\de A}^p} \rB{\frac{\norm{T\de A}_{\text{BMO}}}{C_2}}^p \rB{1 + \frac{C_2}{\norm{T\de A}_{\text{BMO}}}\Lambda} \le \\
            &\le C_{n, M} \rB{\int \modulus{T \de A}^p} \frac{1 + \Lambda}{\Lambda^p}.
        \end{split}
    \end{equation}
    Hence, if we choose $\Lambda$ big enough (depending only on $n$ and $M$) in~\eqref{eq:useless888}, 
    \begin{equation}
        \label{eq:useless999}
        I' \le \frac{1}{2} \int \modulus{T\de A}^p.
    \end{equation}
    Let us now estimate $II$. If $p > \pst$, we can write
    \[
	    \begin{split}
		    \int_{\modulus{T\de A} \le \Lambda} \modulus{T\de A}^p \de x &= \int_{1 < \modulus{T\de A}\le \Lambda} \modulus{T\de A}^p \de x + \sum_{j \ge 0} \int_{2^{-j-1} < \modulus{T\de A} \le 2^{-j}} \le \\
		    &\le C \cB{\Lambda^p \modulus{\de A}^{\pst}(B) + \sum_{j \ge 0} 2^{-(j+1)p} \modulus{\cB{\modulus{T\de A} > 2^{-(j+1)}}} } \le \\
		    &\le C\modulus{\de A}^{\pst}(B)\rB{\Lambda^p + \sum_{j \ge 0} 2^{-j(\pst - p)}} \le \\
		    &\le C(n, p, M) \modulus{\de A}^{\pst}(B),
	    \end{split}
    \]
    which gives~\eqref{eq:rig_LL_p}. In the case $p = \pst$, we are going to make use of the increasing convex function $\Psi$, defined as the linear (convex) continuation of $t \mapsto t^p$ for $t \ge \Lambda$:
    \[
        \Psi(t):=\begin{cases}
            t^{\pst} & \text{if } t \le \Lambda,\\
                    \pst \Lambda^{\pst - 1} t + (1-\pst) \Lambda^{\pst}& \text{if }t \ge \Lambda.
                 \end{cases}
    \]
    \begin{equation}
        \label{eq:uselessB}
        \begin{split}
            II &\le \int_B \Psi(\modulus{T\de A(x)}) \de x \le \int_B \Psi\rB{\fint_B \frac{C\modulus{\de A}(B) \de \modulus{\de A}(y)}{\modulus{x-y}^{n-1}}} \le\\
            &\le \int_B \fint \Psi\rB{\frac{C\modulus{\de A}(B)}{\modulus{x-y}^{n-1}}} \de \modulus{\de A}(y)  \de x = \\
            &= \fint_B \de \modulus{\de A}(y) \int_B  \Psi\rB{\frac{C\modulus{\de A}(B)}{\modulus{x-y}^{n-1}}} \de x \le \\
            &\le \int_{B(0, 2)} \Psi\rB{\frac{C\modulus{\de A}(B)}{\modulus{z}^{n-1}}} \de z = C\int_0^2 \de \rho \rho^{n-1} \Psi\rB{\frac{C\modulus{\de A}(B)}{\rho^{n-1}}} = \\
            &= \int_0^{C\rB{\modulus{\de A}(B)\Lambda^{-1}}^{\frac{1}{n-1}}} \rho^{n-1}\rB{1^*\Lambda^{\pst - 1} \frac{C\modulus{\de A}(B)}{\rho^{n-1}} + (1 - 1^*)\Lambda^{\pst}}\de \rho + \\
            &\quad+ C \int_{C\rB{\modulus{\de A}(B)\Lambda^{-1}}^{\frac{1}{n-1}}}^2 \frac{\modulus{\de A}(B)^{1^*}}{\rho} \de \rho \le \\
            &\le C \modulus{\de A}(B)^{1^*}\rB{1 + \modulus{\log\rB{\modulus{\de A}(B)}}}.
        \end{split}
    \end{equation}
    Combining together~\eqref{eq:useless777}, ~\eqref{eq:useless999} and~\eqref{eq:uselessB}, we obtain~\eqref{eq:rig_LL}.
\end{proof}

\begin{remark}
	The same conclusions can be obtained considering the operator defined by an average on the sphere:
	\[
		\tilde{T}\omega(x):=\int_{\mathbb{S}^{n-1}} \de \mathcal{H}^{n-1}(y) k_y \omega(x).
	\]
\end{remark}
\begin{remark}
  Using Korn's inequality instead of Theorem~\ref{thm:FJM}, one can easily prove the linear counterpart of Theorem~\ref{thm:rig_LL_1}.
\end{remark}

%

  \begin{proposition}
    \label{prop:curl_bounds_D_SOn}
     Let $\Omega \subset \mathbb{R}^n$ be a bounded open set, and suppose $A \in L^2(\Omega)$ and $\spt(A) \Subset \Omega$. Consider a tessellation of $\mathbb{R}^n$ with cubes $\cB{Q^{(\rho)}_i}_{i \ge 1} \equiv \cB{Q(x_i, \rho)}$ of side $\rho$, and define $A_{\rho}$ as the piecewise constant function
     \begin{equation}
      \label{eq:def_A_rho}
      A_{\rho}:=\sum_{i \ge 1} R^{(\rho)}_i \rchi_{Q_{\rho,i}},
     \end{equation}
     where the rotations $R^{(\rho)}_i$ are the ones given by Theorem~\ref{thm:useless1} applied to $A$ on the balls $B(x_i, \frac{3}{2}\rho)$. There exists a constant $C = C(n) > 0$, depending only on the dimension $n$, such that
     \begin{equation}
      \label{eq:prop1_1}
      \frac{1}{\rho}\norm{A - A_{\rho}}_{L^1(\Omega)} + \modulus{DA_{\rho}}(\Omega) \le C\rB{\rho^{\frac{n-2}{2}}\norm{\dist(A, SO(n))}_{L^2(\Omega)} + \modulus{\Curl(A)}(\Omega)}.
     \end{equation}
     In particular, if $A \in SO(n)$ almost everywhere,
     \begin{equation}
      \label{eq:prop1_2}
      \modulus{DA}(\Omega) \le C \modulus{\Curl(A)}(\Omega).
     \end{equation}
     That is, $A \in BV(\Omega, SO(n))$ provided $\modulus{\Curl(A)}(\Omega)$ is finite.
  \end{proposition}
  \begin{proof}
   By definition, the rotations $R^{(\rho)}_i$ in~\eqref{eq:def_A_rho} satisfy
   \[
    \norm{A - R^{(\rho)}_i}_{L^{1^*, \infty}(Q^{(\rho)}_i)} \le C_n\rB{\norm{\dist(A, SO(n))}_{L^{1^*, \infty}} + \modulus{\Curl(A)}(2Q^{(\rho)}_i)}.
   \]
   Let $\phi \in \mathcal{C}^1_c(\Omega)$. Then
   \[
    \modulus{\int A_{\rho}\text{div}(\phi) \de x} \le \sum_{\substack{i, j \text{ s.t. }\\ \partial Q^{(\rho)}_i \cap \partial Q^{(\rho)}_j \ne \emptyset}} \rho^{n-1} \modulus{R^{(\rho)}_i - R^{(\rho)}_j}.
   \]
   Now, for any two adjacent cubes $Q^{(\rho)}_i$ and $Q^{(\rho)}_j$, take the rotation $R'_{\rho, i}$ given applying Theorem~\ref{thm:useless1} to the cube $2Q^{(\rho)}_i$. Then
   \[
    \begin{split}
      \modulus{R^{(\rho)}_i - R^{(\rho)}_j} \rho^{n - 1} &\le \rB{\modulus{R^{(\rho)}_i - R'_{\rho, i}} + \modulus{R'_{\rho, i} - R^{(\rho)}_j}} \rho^{n-1} \le \\
                                                       &\le C_n \rB{\norm{R^{(\rho)}_i - R'_{\rho, i}}_{L^{\pst, \infty}(Q^{(\rho)}_i)} + \norm{R'_{\rho, i} - R^{(\rho)}_j}_{L^{\pst, \infty}(Q^{(\rho)}_j)}} \le \\
                                                        &\le C_n \rB{\norm{A - R^{(\rho)}_i}_{L^{\pst, \infty}(Q^{(\rho)}_i)} + \norm{A - R^{(\rho)}_j}_{L^{\pst, \infty}(Q^{(\rho)}_j)} + \norm{A - R'_{\rho, i}}_{L^{\pst, \infty}(2Q^{(\rho)}_i)}}  \le \\
                                                        &\le C_n\rB{\norm{\dist(A, SO(n))}_{L^{1^*, \infty}(4Q^{\rho}_i)} +\modulus{\Curl(A)}(4Q^{(\rho)}_i)} \\
                                                        &\le C_n\rB{\rho^{\frac{n-2}{2}}\norm{\dist(A, SO(n))}_{L^2(4Q^{\rho}_i)} +\modulus{\Curl(A)}(4Q^{(\rho)}_i)}.
    \end{split}
   \]
   Taking the supremum over $\phi$, since the cubes $4Q^{(\rho)}_i$ overlap only finitely many times, we obtain
   \[
    \modulus{DA_{\rho}}(\Omega) \le C_n\rB{\rho^{\frac{n-2}{2}}\norm{\dist(A, SO(n))}_{\Omega}+ \modulus{\Curl(A)}(\Omega)}.
   \]
   Moreover, from the definition of weak-$L^1$:
   \[
    \frac{1}{\rho} \int_{Q^{(\rho)}_i} \modulus{A - A_{\rho}} \de x \le C_n \norm{A - A_{\rho}}_{L^{\pst, \infty}(Q^{(\rho)}_i)} \le C_n\rB{\norm{\dist(A, SO(n))}_{L^{1^*, \infty}(4Q^{\rho}_i)} + \modulus{\Curl(A)}(4Q^{(\rho)}_i)}.
   \]
   This gives in particular~\eqref{eq:prop1_1}. Moreover
   \[
    \begin{split}
	    \norm{A - A_{\rho}}_{L^1(\Omega)} &\le \sum_{i\ge 1} \norm{A - A_{\rho}}_{L^1(Q^{(\rho)}_i)} \le C_n \rho \sum_{i \ge 1}\rB{\norm{A - A_{\rho}}_{L^{\pst, \infty}(2Q^{(\rho)}_i)} + \modulus{\Curl(A)}(2Q^{(\rho)}_i) }\le\\
	    &\le C_n \rho \rB{\norm{A - A_{\rho}}_{L^{\pst, \infty}(\Omega)} + \modulus{\Curl(A)}(\Omega)} \xrightarrow[\rho \to 0]{} 0.
    \end{split}
   \]
   That is, $A_{\rho} \to A$ strongly in $L^1$. Thus, if we let $\rho\to 0$, we obtain~\eqref{eq:prop1_2} provided $A \in SO(n)$ almost everywhere.
  \end{proof}

\end{document}